\documentclass[11 pt]{amsart}

\usepackage{amssymb, amsmath, amsthm, amsfonts}
\usepackage{verbatim}
\usepackage{bbm}
\usepackage{enumerate}
\usepackage{dsfont}
\usepackage{upgreek}
\usepackage[mathscr]{eucal}
\usepackage{tikz}
\usepackage[normalem]{ulem}

\usepackage{enumitem}

\usepackage{textcomp}

\usepackage[backgroundcolor=black!5,linecolor=black]{todonotes}
\usepackage{comment}

\usepackage[explicit,indentafter]{titlesec}
\titleformat{\section}{\centering\normalfont\scshape}{\thesection.}{.5em}{#1}
\titleformat{\subsection}[runin]{\normalfont\itshape}{\textnormal{\thesubsection.}}{.5em}{#1.}
\titleformat{\subsubsection}[runin]{\normalfont\itshape}{\thesubsubsection.}{.5em}{#1.}
\titlespacing{\section}{0em}{1em}{0.5em}
\titlespacing{\subsection}{0em}{.5em}{0.5em}

\usepackage{stackengine}
\newcommand{\ubar}[1]{\stackunder[1.2pt]{$#1$}{\rule{.9ex}{.075ex}}}

\newcommand{\ox}{{\overline x}}
\newcommand{\oy}{{\overline y}}

\newcommand{\ux}{{\ubar x}}
\newcommand{\uy}{{\ubar y}}

\newcommand{\om}{\omega}
\usepackage{color}
\definecolor{gray}{gray}{0.5}

\newcommand{\cmt}[1]{}

\newcommand{\vertiii}[1]{{\left\vert\kern-0.25ex\left\vert\kern-0.25ex\left\vert #1 
    \right\vert\kern-0.25ex\right\vert\kern-0.25ex\right\vert}}
    
    \newcommand{\detail}[1]{}

	\newcommand{\drawQbg}{
		\fill (Q1) node [left] {$Q_1$} circle [radius=.02em];
		\fill (Q2) node [above left] {$Q_{2}$} circle [radius=\ptsize];
		\fill (Q3) node [right] {$Q_{3}$} circle [radius=\ptsize];
		\fill (Q4g) node [below right] {$Q_{4}$} circle [radius=\ptsize];
		\fill [color=\QbgFillColor,opacity=\QbgFillOpacity] (Q1) -- (Q2) -- (Q3) -- (Q4g) -- cycle;
		\draw [opacity=\QbgLineOpacity] (Q1) -- (Q2) -- (Q3);
		\draw [opacity=\QbgLineOpacity] (Q4g) -- (Q1);
		\if\QbgDrawCritSeg1
			\draw [style=\QbgCritSegStyle,opacity=\QbgCritSegOpacity] (Q3) -- (Q4g);
		\fi
	}

\def\lc{\lesssim}
\def\gc{\gtrsim}

\def\sH{\mathscr{H}}

\NewDocumentEnvironment{amatrix}{>{\SplitArgument{1}{|}}m}
 {\left(\makeamatrix#1}
 {\end{array}\right)}
\NewDocumentCommand{\makeamatrix}{mm}{%
  \IfNoValueTF{#2}
    {\begin{array}{@{}*{#1}{c}@{}}}
    {\begin{array}{@{}*{#1}{c}|*{#2}{c}@{}}}%
}

\makeatletter
\def\widebreve{\mathpalette\wide@breve}
\def\wide@breve#1#2{\sbox\z@{$#1#2$}%
     \mathop{\vbox{\m@th\ialign{##\crcr
\kern0.08em\brevefill#1{0.8\wd\z@}\crcr\noalign{\nointerlineskip}%
                    $\hss#1#2\hss$\crcr}}}\limits}
\def\brevefill#1#2{$\m@th\sbox\tw@{$#1($}%
  \hss\resizebox{#2}{\wd\tw@}{\rotatebox[origin=c]{90}{\upshape(}}\hss$}
\makeatletter

\def\eps{\varepsilon}
\def\bbone{{\mathbbm 1}}

\newcommand{\wt}{\widetilde}

\newcommand{\Be}{\begin{equation}}
\newcommand{\Ee}{\end{equation}}

\newcommand{\Bm}{\begin{multline}}
\newcommand{\Em}{\end{multline}}

\def\intslash{\rlap{\kern  .32em $\mspace {.5mu}\backslash$ }\int}
\def\qsl{{\rlap{\kern  .32em $\mspace {.5mu}\backslash$ }\int_{Q_x}}}

\def\lc{\lesssim}
\def\gc{\gtrsim}

\def\emph#1{{\it #1 }}

\def\Ga{\Gamma}
\def\ga{\gamma}

\def\rank{{\mathrm{rank}}}

\def\supp{{\mathrm{supp}}}

\def\range{{\mathrm{range}}}

\def\inn#1#2{\langle#1,#2\rangle}

\def\noi{\noindent}

\def\ga{\gamma}             \def\Ga{\Gamma}
\def\eps{\varepsilon}

\def\la{\lambda}             \def\La{\Lambda}

\def\om{\omega}              \def\Om{\Omega}
\def\vth{\vartheta}

\def\fM{{\mathfrak {M}}}

\def\fg{{\mathfrak {g}}}

\def\bbH{{\mathbb {H}}}

\def\bbR{{\mathbb {R}}}

\def\bbV{{\mathbb {V}}}

\def\cA{{\mathcal {A}}}

\def\cE{{\mathcal {E}}}
\def\cF{{\mathcal {F}}}

\def\cJ{{\mathcal {J}}}

\def\cM{{\mathcal {M}}}

\def\cV{{\mathcal {V}}}
\def\cW{{\mathcal {W}}}

\def\emph#1{{\it #1}}
\def\textbf#1{{\bf #1}}

\def\beq{\begin{equation}}
\def\endeq{\end{equation}}

\def\bs{\begin{split}}
\def\es{\end{split}}

\theoremstyle{plain}

\newtheorem{thm}{Theorem}[section]
\newtheorem{prop}[thm]{Proposition}

\newtheorem*{thm*}{Theorem}
\newtheorem*{conj*}{Conjecture}
\newtheorem*{openproblem*}{Open Problem}

\theoremstyle{remark}

\newtheorem*{remarksa}{Remarks}

\numberwithin{equation}{section}

\definecolor{jrcol}{rgb}{0,0,1.} 
\definecolor{ascol}{rgb}{1.,0,0} 

\def\R{\mathbb{R}}

\begin{document}
\title[On perturbed   Nevo--Thangavelu  means]{Failure of stability of a maximal operator bound for perturbed Nevo--Thangavelu means}

\author[Jaehyeon Ryu and Andreas Seeger]{Jaehyeon Ryu \ \ \ \ \ Andreas Seeger} 

\address{Jaehyeon Ryu: Department of Mathematics, Ewha Womans University, Seoul 03760, Korea}
\email{jhryu67@ewha.ac.kr}

\address{Andreas Seeger: Department of Mathematics, University of Wisconsin, 480 Lincoln Drive, Madison, WI, 53706, USA.}
\email{seeger@math.wisc.edu}

\begin{thanks}
    {J.R.  was supported in part by the Ewha Womans University Research Grant
of 2025.  
A.S.  was  supported in part by NSF Grant DMS-2348797.} 
\end{thanks}


\begin{abstract}
Let $G$ be  a two-step nilpotent Lie group,  identified via the exponential map with the Lie-algebra $\mathfrak g=\mathfrak g_1\oplus\mathfrak g_2$, where $[\mathfrak g,\mathfrak g]\subset \mathfrak g_2$.
We consider    maximal functions associated to spheres in a $d$-dimensional linear subspace  $H$, dilated by the automorphic dilations. $L^p$ boundedness results for the case where $H=\mathfrak g_1$ are  well understood. Here   we consider the case of  a tilted hyperplane  $H\neq \mathfrak g_1$ which is not  invariant under the automorphic dilations. In the case of M\'etivier groups it is known that the $L^p$-boundedness results are stable under a small linear tilt. We show that this  is generally not the case for other two-step   groups, and provide  new necessary conditions for $L^p$ boundedness. We prove these results in a more general setting with  tilted versions of  submanifolds of $\mathfrak g_1$. 
\end{abstract}

\maketitle

\section{Introduction}\label{sec:intro} 
Let $G$ be a finite dimensional two-step nilpotent Lie group, which via the exponential map we identify with its Lie algebra $\fg$.
We fix  the direct sum decomposition $\fg=\fg_1+\fg_2$ where $\dim\fg_1=d$, $\dim \fg_2=m$,  $[\fg,\fg]\subset \fg_2$  and $\fg_2$ is contained in the center.  By the Baker-Campbell-Hausdorff formula and the two-step  assumption the group law is given by
\[(X,U)\cdot (Y,V)= (X+Y, U+V +\tfrac 12 [X,Y])\] where the commutator $[X,Y]$ belongs to $\fg_2$.  On $G$ a natural group of automorphic dilations is given for $t>0$ by $\delta_t: (X,U)\mapsto (tX, t^2U)$. For every linear functional  $\vth\in \fg_2^*$, and for $X\in \fg_1$  the functional $\cJ_X^\vth: Y\to \frac{1}{2} \vth([X,Y])$ belongs to $\fg_1^*$ and depends linearly on $X$ and $\vth$. The linear map $\cJ^\vth: \fg_1\to \fg_1^*$ given by $\cJ^\vth[X]=\cJ^\vth_X$ depends linearly on $\vth$ and can be identified with the bilinear form $(X,Y)\mapsto \frac 12 \vth([X,Y])$. 

Let $H$ be a $d$-dimensional plane which is transversal to $\fg_2$, i.e. $H$ is given as \[H=\{(X, \La(X)),\, X\in \fg_1\}\] and $\Lambda:\fg_1\to \fg_2$  is linear. We assume that a standard scalar product is defined on $\fg_1$ and 
define  a measure $\varsigma^\La $ by 
\[\inn{f}{\varsigma^\La} = \frac 1{\varsigma(S^{d-1}) }\int_{S^{d-1}}  f(X,\La(X) )  d\varsigma   \]
where $\varsigma$ denotes surface  measure on the unit sphere $S^{d-1}=\{X:|X|=1\}$.   
Define the dilate $\varsigma^\Lambda _t$ of $\varsigma^\La$ by $\inn{f}{\varsigma^\La_t} = \int f(tX,t^2U)d\varsigma^\La$. For Schwartz functions $f\in \mathcal S(\fg)$  we are then interested in the maximal function generated by the  non-commutative convolutions $f*\varsigma_t$. 
That is, we ask for $L^p$-boundedness properties of  the maximal operator $\cM^\La$, defined by  
\[  \cM^\La f =\sup_{t>0} |f*\varsigma^\La_t| . \]

This problem was first proposed by Nevo and Thangavelu in \cite{NevoThangavelu1997}, for the 
Heisenberg groups, with $\La=0$ and $\Sigma$ the sphere $\{|X|=1\}$ in $\fg_1$. 
For $d\ge 3$,  and $\La=0$ an optimal boundedness result was proved  by M\"uller and the second author in \cite{MuellerSeeger2004} for M\'etivier groups (i.e. in the case that  for every nonzero $\vth\in \fg_2^*$ the linear map $\cJ^\vth:\fg_1\to \fg_1^*$ is an isomorphism); in this case the maximal operator is  bounded on $L^p(G)$ if and only if  $p>\frac{d}{d-1} $. Independently, this result  was also obtained for the Heisenberg groups $ \bbH_n$ ($n\ge 2$) 
 by Narayanan and Thangavelu \cite{NarayananThangavelu2004},  using a different approach.  More recently we showed in \cite{RyuSeeger} that for $\La=0$ the maximal operator $\cM^0$ 
 is $L^p$ bounded for $p>\frac{d}{d-1}$ on all two-step groups with $d>2$.   The $L^p$-boundedness in the case $d=2$, in particular the case of the Heisenberg group $\bbH_1$,  remains open (see however positive results for the case of Heisenberg-radial functions in \cite{BeltranGuoHickmanSeeger}, \cite{LeeLee}). 

The question about $L^p$-boundedness of the perturbed maximal operators $\cM^\La$ was first raised by M\"uller and the second author in \cite{MuellerSeeger2004}. 
Satisfactory results for the case of M\'etivier groups were obtained using an $L^2$ local smoothing estimate 
in \cite{RoosSeegerSrivastava-imrn} (see also related results in \cite{KimJoonil}); for another approach on the Heisenberg groups based on fixed time $L^p$-regularity results via decoupling see \cite{PramanikSeeger, AndersonCladekPramanikSeeger}. We remark that the condition $p>\frac{d}{d-1}$ is always necessary for $L^p$-boundedness. 
This follows from a variant of Stein's example \cite{SteinPNAS1976} (as
 was already noted in \cite{LiuYan, RyuSeeger}). 
It turns out that on the  M\'etivier groups the above results for $\La= 0$  remain true under small perturbations \cite{RoosSeegerSrivastava-imrn}. That is,   for small $\|\La\|$ the operator $\cM^\La$,  is still $L^p$ bounded for $p>\frac{d}{d-1}$ when $d>2$. The purpose of this paper is to investigate this stability phenomenon for other two-step groups,  outside the M\'etivier class. One might have expected
that 
$L^p$ boundedness of $\cM^\La$ 
with small but nonzero $\La$  still holds for all two-step   groups  when $p>\frac{d}{d-1}$, but we show in this paper that this is not the case. 

We investigate this phenomenon in a more general setting. 
In what follows we fix a nontrivial, nonnegative  $\chi\in C^\infty_c(\fg_1)$ 
and surface measure $d\sigma$ on a $k$-dimensional $C^1$-submanifold $\Sigma$ of $\fg_1$.
Let $d\mu= \chi d\sigma$.  For $t>0$ we define the measure $\mu_t$ by 
\[ \inn{f}{\mu_t} =\int f(tX, t^2 \Lambda (X) ) \chi(X) d\sigma(X).\] We are then interested in lower bounds for the maximal function $\fM f= \sup_{t>0} |f*\mu_t|$;
these follow from lower bounds for a localized operator. 
Let $I\subset (0,\infty)$ a compact interval of positive length and define, for continuous functions $f$ with compact support, the maximal operator $\fM_I\equiv \fM_I^\La$ by 
\Be \label{eq:defofmaxop}  \fM_I f(X,U)= \sup_{t\in I} |f*\mu_t(X,U) |. \Ee 

For $\omega^\circ\in \Sigma$ let 
$ (T_{\omega^\circ}\Sigma)^0=\{ \phi\in \fg_1^*: \,\phi(v)=0 \text{ for all }  v\in T_{\omega^\circ} \Sigma\}, $  
the annihilator of the tangent space $T_{\om^\circ}\Sigma.$
For $\vartheta\in \fg_2^*$ consider 
\Be\label{Vlatheta}  \bbV_{\La,\vth} = \mathrm{range} (\cJ^\vth) + \bbR (\vartheta\circ\La) \,.
\Ee 
which is a linear subspace of $\fg_1^*$. 
Our unboundedness results rely on the following hypothesis, for $1\le r\le k$.

\medskip 
\noi{\it Hypothesis $\sH(r)$.} {\it
There exist $\vartheta\in \fg_2^*$ and  $\omega^\circ\in \Sigma$ satisfying 
 $\chi(\om^\circ)\neq 0$ such that
\begin{itemize}
\item $\vartheta\circ\La\neq 0$.

\item $\bbV_{\La,\theta} \cap (T_{\omega^\circ} \Sigma)^0=\{0\}$

\item $\dim (\bbV_{\La,\vartheta}) =r$. 

\end{itemize} 
}

Note that the second condition implies that there exists a scalar product on $\fg_1^*$ with respect to which  $\bbV_{\la,\theta}$ is orthogonal to  $(T_{\omega^\circ} \Sigma)^0$, and  the condition means that $\bbV_{\la,\theta}$ is a subspace of 
$((T_{\omega^\circ} \Sigma)^0)^\perp$ which is identified with $T^*_{\om^\circ} \Sigma$. 
\medskip

\begin{thm}\label{thm:main} Assume  $d\ge 2$,  $\La\neq 0$, $1\le r\le k\le d$  and   that Hypothesis $\sH(r)$ holds. 
    \begin{itemize}[topsep = 2pt, itemsep = 2pt]
        \item [(i)] If  $2\le r\le k$ and 
     $\fM_I$ extends to a bounded operator on $L^p(G)$ then  $p>\frac{r+1}r$.
        \item [(ii)] If $r=1$ and  $\fM_I$ extends to a bounded operator on $L^p(G)$ then $p = \infty$.
    \end{itemize}
\end{thm}

\begin{remarksa} \hfill

(a) Because of the skew-symmetry of the bilinear form $(X,Y)\mapsto \vth([X,Y])$ the dimension of range $\cJ^\vth$ is always even. That is, in part (ii) of Theorem \ref{thm:main} where $r=1\le k\le d$ we get unboundedness for all $p<\infty$ if there exists $\vth\in \fg_2^*$ such that $\vth\circ \La\neq 0$ and $\cJ^\vth=0$. For $m=1$ this corresponds to the commutative Euclidean case. The main tool used in the proof of this unboundedness result 
is the (complement of the) Nikodym set in the plane  \cite{Nikodym, Guzman1}, and we exploit  that in some situations the nonisotropic dilations in some compact $t$-interval have an effect which is similar to rotations.

(b) Theorem \ref{thm:main} applies with $k=d$ if $\La\neq 0$ and $d\mu= \chi dX$ where $dX$  is Lebesgue measure on $\fg_1$.

(c) Theorem \ref{thm:main} applies with $k=d-1$ if $d\mu=\chi d\sigma$ where $\sigma$ is surface measure   on the smooth boundary of a convex domain $\Omega$ in $\fg_1$.  
In this case every linear hyperplane is the cotangent space to $\partial\Om$ at some point $\omega^\circ$, i.e.,
$\bbV_{\La,\theta}$ is a linear subspace of one of those cotangent spaces. We note that $\fM_I$  fails to be bounded for $p\le \frac{d}{d-1}$ as can be seen by a variant of the familiar example by Stein \cite{SteinPNAS1976}. The new examples proving the necessity of the condition $p>\frac{r+1}{r}$ in Theorem \ref{thm:main}  are thus relevant for the cases $2\le r\le d-2$. 

(d) In case (c) above, under the additional assumption that $0$ is in the interior of $\Omega$ and that  $\partial \Om$ has nonvanishing curvature everywhere, we have  a tilted version of the setup in our previous paper \cite{RyuSeeger}.  Theorem \ref{thm:main} shows that the hypothesis $\La=0$ in the general result in \cite{RyuSeeger} cannot be dropped. 
In a  subsequent paper  we intend to prove  
 satisfactory upper bounds for the cases $r\le d-2$, under the nonvanishing Gaussian curvature assumption on $\partial\Om$.

 (e) Again for the case (c), we conjecture that stability holds in the cases $r=d-1$ and $r=d$; this appears to be a difficult problem. Moreover, even in the M\'etivier case (where $r=d$ since the $\cJ^\vartheta$ are invertible for $\vartheta\neq 0$) the $L^p$ boundedness for $p>\frac{d}{d-1}$ was established in \cite{RoosSeegerSrivastava-imrn} only for sufficiently small $\|\La\|$, and it would be interesting to settle the general case. 

   \end{remarksa}

\noi{\it Outline.} In \S\ref{sec:reduction}  we discuss coordinates on the group and show that it suffices to prove the lower bounds in the case $m=1$. In \S\ref{sec:large-r} we show unboundedness for $p\le \frac{r+1}{r}$ and in \S\ref{sec:r=1} we treat the special case $r=1$. 

\medskip

\noi{\it Acknowledgements.} We thank the referee for a careful reading of the paper and useful suggestions. 
    {J.R.  was supported in part by the Ewha Womans University Research Grant
of 2025. 
A.S.  was  supported in part by NSF Grant 2348797.

\section{Preliminary reductions}\label{sec:reduction}

\subsection{Coordinates on $G$} 
\label{sec:coordinates} 
Choosing coordinates on $\fg$ we may identify   $\fg_1$ with $\bbR^d$ and $\fg_2$ with $\bbR^m$. We denote coordinates $x$ on $G$ by $x=(\ubar x, \overline x)\in \bbR^d\times \bbR^m$; then the group law becomes 
\[
(\ux,\ox) \cdot (\uy,\oy) = (\ux+\uy, \ox+\oy + \ux^\intercal \vec J\uy),
\]
where $\ux^\intercal \vec J\uy = (\ux^\intercal J_1 \uy,\dots \ux^\intercal J_m \uy)$ and $J_1,\dots, J_m$ are $d\times d$ skew-symmetric matrices. Our convolution operator is then written as 
\Be\label{eq:Atdef}
A f(x,t) := f\ast \mu_t(x) = \int f(\ux - t\om, \ox - t\ux^\intercal \vec J \om - t^2 \La \om) d\mu(\om).
\Ee
The linear map $\La:\bbR^d\to \bbR^m$ is given by $\La \uy= \sum_{i=1}^m (\la_i^\intercal \uy)  e_{d+i}$  where $ e_{d+1},\dots,e_{d+m}$ is the standard basis of $\bbR^m$, and $\la_i\in \bbR^d$, $i=1,\dots, m$. 
For $1\le i\le m$, let $S_i$ be the  $(d+1)\times d$ matrix given by
\begin{align} 
\label{def:Si}
    S_i = \begin{pmatrix}
    J_i \\ \lambda_i^\intercal
\end{pmatrix}.
\end{align} If $\vth\in \fg_2^*$ is given by $\vth(e_{d+i} )=\theta_i$, $i=1,\dots, m$, then the dimension of the space $\bbV_{\La,\vth}$  in Theorem \ref{thm:main} is equal to 
\begin{align}\label{def:redef=rtheta}
r(\vth) = \rank \big(\sum_{i=1}^m\theta_iS_i\big).
\end{align}

\subsection{Scaling} \label{sec:parabolicscaling} A calculation shows that $f*\mu_{st} (x)= [f(\delta_s\cdot)]*\mu_t (\delta_{1/s} x) $ and thus $\fM_{sI} f(x)=\fM_I[f(\delta_s\cdot)] (\delta_{1/s}x)$ which shows $\|\fM_{sI}\|_{L^p\to L^p} = \|\fM_I \|_{L^p\to L^p}$.  Since for $I_1\subset I_2$ we have $\|\fM_{I_2}\|_{L^p\to L^p} \ge \|\fM_{I_1} \|_{L^p\to L^p}$ we may assume,  after a finite decomposition of a $t$-interval and scaling,    that $I\subseteq [1-\eps, 1+\eps] $ for some small  $\eps>0$.

\subsection{Reduction to the case $m=1$} 
\label{sec:reductiontom=1}
We show that in order to prove the unboundedness results in  Theorem \ref{thm:main} it suffices to do this for the case $m=1$. 

Let $G$ be a general two-step nilpotent group of dimensions $d+m$, $m>1$, which we have identified with $\bbR^d\times \bbR^m$ as above. By the definition of $r$ in Hypothesis $\sH(r)$, there exists $\theta = (\theta_1,\dots, \theta_m)\in \mathbb S^{m-1}\setminus \{0\}$ 
such that $\rank(\sum_{i=1}^m \theta_i S_i) = r$   and $\sum_{i=1}^m \theta_i \lambda_i^\intercal\neq 0$. Choose unit vectors $(b^i)_{2\le i\le m}$ in $\bbR^m$ such that $\theta$, $b^2, \dots, b^m$ are mutually orthogonal. Let $R_\theta$ be the $m\times m$ rotation matrix with columns $\theta, b^2,\dots, b^m$. 
Then the averaging operator $A_t$ is expressed as
\Be \label{Atrot}
A_t f(\ux, \ox) = \int f\big(\ux-t\om, R_\theta \big(R_\theta ^\intercal \ox - t R_\theta ^\intercal \big((\ux,t)^\intercal \vec S \om\big)\big)\big) d\mu(\om)
\Ee  with $(\ux,t)^\intercal \vec S \om=\sum_{i=1}^m (\ux^\intercal J_i\om+t\la_i^\intercal \om) e_{d+i}$. 
Note that \[\inn{R_\theta ^\intercal \big((\ux,t)^\intercal \vec S \om\big)}{e_{d+1}} = (\ux,t)^\intercal (\sum_{i=1}^m \theta_i S_i)\om. \]  
Thus, after a  conjugation with a rotation in $\bbR^m$,  we may assume that \[\text{ $R_\theta  = I_m$,  $\rank(S_1) = r$ and $\la_1^\intercal \neq 0$.}\]

Let $I\subset (0,\infty)$ be a compact subinterval. We now consider a class of compactly supported smooth functions $f=\widetilde f\otimes h$ given by \[f(\uy, \oy_1, \oy') = \wt f(\uy, \oy_1) h(\oy'),\] with  $\wt f\in C_c^\infty(\R^{d+1})$, $h\in C_c^\infty(\R^{m-1})$. For $f$ in this class we have
\[ A_t f(x)= \int \widetilde f(\ubar x-t\om, x_{d+1}-t(x,t)^\intercal S_1\om) h(\sum_{i=2}^m (\ox_i - t (\ux, t)^\intercal S_i \om) e_{d+i})  d\mu(\om)\] with $t(x,t)^\intercal S_1\om) =
t \ux^\intercal J_1\om +t^2\la_1^\intercal \om$. 
Define 
\Be\label{tildeAt} \wt A_t \wt f(\ubar x,x_{d+1})= \int \wt f(\ubar x-t\om, x_{d+1} -t(x,t)^\intercal S_1\om)
d\mu(\om) .\Ee

Let $\widetilde B$ be a ball in  $\bbR^{d+1}$ and $B'$ be a ball in  $\bbR^{m-1}$. Choose $h\in C^\infty_c(\bbR^{m-1})$ 
so that $h\equiv 1$ on a  large compact set, specifically
\[h(\sum_{i=2}^m (\ox_i - t (\ux, t)^\intercal S_i \om) e_{d+i}) = 1\] for every $\om \in \supp(\mu)$, $t\in I$, $x\in \widetilde B\times B'$.
Then 
\begin{align}\label{eq:sliceineq}
    \big\|\bbone_{\widetilde B\times B'}\sup_{t\in I}\big|A_t f\big|\big\|_{L^p(\bbR^{d+m})} = |B'|^{1/p} \big\|\bbone _{\wt B}\sup_{t\in I}\big|\wt A_t \wt f\big|\big\|_{L^p(\bbR^{d+1} )},
\end{align}

Denote by  $\widetilde G$ denote the $d+1$-dimensional two-step group with group law 
$(\ubar x, x_{d+1})\cdot (\ubar y, y_{d+1}) =(\ubar x+\ubar y, x_{d+1}+y_{d+1}+\ubar x^\intercal J_1\ubar y)$.
From \eqref{eq:sliceineq}  we conclude that the desired unboundedness  on   $G$ would follow once on  $\wt G$, the local maximal operator $g\to \bbone_{\wt B}\sup_{t\in I}|\wt A_t g|$ is shown to be unbounded on $L^p(\wt G)$; here  $\wt B\subset \wt G$ is a suitable ball, and   $I\subset (0,\infty)$ is a compact interval. Below we will consider  $\sup_{t\in I} |A_t g|$ for compactly supported functions $g$ so that the maximal function is then  supported on compact sets; thus the characteristic function of the ball  $\wt B$ will can then be dropped in  the definition of the local maximal function.

\section{Unboundedness for $p\le \frac{r+1}{r}$} \label{sec:large-r}
In what follows we prove that our (local) maximal operator is not bounded on $L^p$ if $p\le \frac{r+1}{r}$, showing part (i) of Theorem \ref{thm:main}, for $r\ge 2$.  By the reduction in \S\ref{sec:reductiontom=1} we  may assume $m=1$, $\La(\ubar x)=\la^\intercal \ubar x$ where $\la$ is a nonzero vector in $\bbR^d$. We thus write $J=J_1$ and $\la=\la_1$ in \eqref{tildeAt} and consider the maximal function  $\fM_I g=\sup_{t\in I}|A_t g|$ where
$I=[1-\eps_0, 1+\eps_0]$ with $\eps_0<\frac 12$, 
\Be \label{eq:Atg} A_t g(\ubar x,x_{d+1})= \int g(\ubar x-t\om, x_{d+1} -t \ux^\intercal J\om -t^2\la^\intercal \om) 
\chi(\om) 
d\sigma(\om) .\Ee  
Moreover $S^\intercal=\begin{pmatrix} J^\intercal &\la\end{pmatrix}$ and $r$ is the rank of $S^\intercal$. We have $0\neq \la\in \range (S^\intercal)$. Let $\Pi:\bbR^d\to \bbR^d$ be the orthogonal projection to the range of $S^\intercal$.  

By assumption, there exists  $\omega^\circ\in \Sigma$ such that  $\chi(\om^\circ)\neq 0$,  and $\range(S^\intercal)$ is contained in the tangent space $T_{\omega^\circ} \Sigma$. Let $\widetilde \Sigma$ be a neighborhood (in $\Sigma$) of $\omega^\circ$ such that $|\chi(\om)|\ge c>0$ for $\om\in \widetilde \Sigma$. By choosing $\widetilde \Sigma$ sufficiently small we may assume that there is a parametrization of $\widetilde \Sigma$ 
with $u\in \bbR^{k}$ close to the origin and $\Gamma(0)=\omega^\circ$ such that $\{\frac{\partial \Gamma}{\partial u_i} (0)\}_{i=1}^{k}$ is an orthogonal basis of $T_{\omega^\circ}\Sigma$, and such that $\frac{\partial \Gamma}{\partial u_1}(0) = \frac{\lambda}{|\lambda|}$ 
and  
$\{\frac{\partial \Gamma}{\partial u_i}(0)\}_{i=1}^r$ is a basis of $\range (S^\intercal)$.
We split the parameters as $u=(u',u'')\in \bbR^r\times \bbR^{k-r}$; in the case $k=r$ the variable $u''$ is not present (which requires a slight notational modification in what follows).

We use the implicit function theorem to solve for fixed $x,t$ the equation $\Pi (\ubar x-t\omega)=0$. 
More precisely we solve for  $u'$  the  equation 
\Be\label{eq:impl}  \Pi( \ubar x- t \Gamma(u', u'') )=0,\Ee
for $t$ near $1$, for $\ubar x$  near $\om^\circ$ and for $u''=(u_{r+1},\dots, u_{k} )$ near $0\in \bbR^{k-r}$. This is possible since $\ubar x- t\Gamma(u)|_{(\underline x,t,u)=(\om^\circ, 1,0)} =0$  and 
\[ \frac{\partial}{\partial u_i} \big( \ubar x-t\Gamma(u))|_{u=0} = -t \frac{\partial\Gamma}{\partial u_i}(0),\quad  i=1,\dots, r\] 
are linearly independent. 
As a consequence we solve \eqref{eq:impl} by $u' = h( \ubar x,t,u'')$  for $(\ubar x,t, u')$ near $(\om^\circ,1,0)$. Introducing coordinates $u'= h(\ubar x,t,u'')+\beta'$  for small $\beta'$ there is a small $\eps>0$,  with $\eps<\eps_0$, such that for $\delta\ll \eps$ and positive constants $c_1$, $c_2$ the 
measure of the set $\{u\in \bbR^{k}: c_1\delta/2\le |u'-h(\ubar x,t, u'') | \le c_1\delta, |u''|\le c_2\eps \}$ 
is bounded below by $c \eps^{k-r} \delta^r$, for 
$|\ubar x-\om^\circ|\le \eps$, $|t-1|\le \eps$. 
Consequently,
the surface measure of 
\Be\label{eq:Wdelta} \cW_\delta(\ubar x,t)= \big\{\omega\in \widetilde \Sigma:   \frac \delta 2<|\Pi (\ubar x-t\om )|\le \delta, |\om-\om^\circ|\le \eps \big\} \Ee
satisfies 
\Be\label{eq:lowerbd-surface} \inf \big\{ \sigma\big(\cW_\delta(\ubar x,t) \big) :\text{   $|\ubar x-\om^\circ|\le \eps$, $|t-1|\le \eps$}\big\} \, 
\ge \,c_\eps\delta^r. \Ee
 We  choose $\eps$ so small such that $\chi(\ubar x)\neq 0$ for $|\ubar  x-\omega^\circ|\le \eps$.

In what follows we fix $\eps>0$ such that  \[\eps \ll \tfrac 14  (1+|\la|+ \|J\| +|\omega^\circ|+\|J\||\om^\circ|) ^{-1} , \]  and  work with a parameter  $\delta\ll \eps$. 
Define
\begin{align*} R_{\delta}&=\{(\ubar y, y_{d+1}):  \frac \delta 2 \le |\Pi(\ubar y)|\le \delta,\, |y_{d+1}| \le \eps^{-1} \delta, \, 
|\ubar y|\le 1\}, 
\\ g_\delta&= \bbone_{R_\delta} . \end{align*} 
We test the maximal operator on $g_\delta$.

Let $m_*= \max \{ |\la^\intercal \ubar  x|: |\ubar x-\omega^\circ|\le \eps\}$. Then $m_*=|\la^\intercal \omega^\circ|+\eps|\la|$  which lies between $\eps|\la|$ and $(|\omega^\circ|+2\eps)|\la|$. 
Let 
\[V_\eps=\{ \ubar x\in \bbR^d: |\ubar x-\omega^\circ| \le \eps,\,   m_*/2\le |\la^\intercal \ubar x|\le m_*\}.\] 
Then $|V_\eps|>0.$  Let 
\[U_\eps=\Big \{(\ubar x, x_{d+1}): \ubar x\in V_\eps, \, 1-\frac\eps 2 < \frac{x_{d+1}}{|\la^\intercal \ubar x|} <1+\frac \eps 2\Big\}\,.\]  Observe 
\[ |U_\eps|=  \int_{ V_\eps} \eps|\la^\intercal \ubar x| d\ubar x\ge \frac{m_*\eps |V_\eps| }{2}>0.\]

We wish 
to derive a lower bound for $A_t f(x)$, for $x\in U_\eps$ and suitable $t=t_x\in (1-\eps, 1+\eps)$.
Observe that 
 \[x_{d+1}-t^2\la^\intercal \om - t \ubar x^\intercal J\om= x_{d+1} - t \lambda^\intercal \ubar x + 
 (t\la-J\ubar x)^\intercal (\ubar x-t\om) .
 \] 
It is natural to choose   \[ t_x = \frac{x_{d+1}} {\lambda^\intercal \ubar x}
\]  which for $x\in U_\eps$ lies  in $(1-\eps, 1+\eps)$ and gives 
 \begin{align*}  x_{d+1}-t_x^2\la^\intercal \om - t_x  \ubar x^\intercal J\om
 &=   
 (t_x\la-J\ubar x)^\intercal (\ubar x-t_x\om) 
 \\&=
  (t_x\la-J\ubar x)^\intercal \Pi(\ubar x-t_x\om) 
 . 
 \end{align*} 
 Note that  $|t_x\la-J\underline x|\le (1+\eps)|\la|+\|J\|(|\om^\circ|+\eps) \le \eps^{-1}$. 
For $\omega \in \cW_\delta(\ubar x, t_x)$  we get $\frac \delta 2\le |\Pi(\ubar x-t_x\om)|\le \delta$  and hence 
\[ |x_{d+1}-t_x^2\la^\intercal \om - t_x  \ubar x^\intercal J\om|
\le  \eps^{-1} |\Pi (\ubar x-t_x\om) | \le \eps^{-1}\delta.
\]
Also,  
$|\ubar x-t_x\om| \le |\ubar x-\om^\circ|+|\om^\circ(1-t_x)| +t_x|\om^\circ-\om| 
\le \eps+\eps|\om^\circ|+ 2\eps <1$. We have thus shown that 
$(\ubar x-t_x\om, x_{d+1}-t_x^2\la^\intercal \om - t_x \ubar x^\intercal J\om)\in R_\delta$ for $x\in U_\eps$ and $\om\in W_{\delta}(\ubar x, t_x)$ and consequently 
we obtain
\[ \fM_I g_\delta  (x)\ge A_{t_x} g_\delta (x) \ge c\sigma(W_\delta(\ubar x, t_x)) \ge c_\eps c\delta^r 
\text{ for $x\in U_\eps$.}\]
Hence $\|\fM_I g_\delta\|_{L^p} \gc_\eps \delta^r |U_\eps|^{1/p} $ and since $\|g_\delta\|_p=|R_\delta|^{1/p}\lc_\eps \delta^{\frac{r+1}{p}} $ we obtain
\[\frac{\|\fM_I g_\delta  \|_{L^p(U_\eps)}}{\|g_\delta\|_p }
\gc_\eps  \delta^{r- (r+1)/p}.\] Letting $\delta\to 0$ shows that $\fM_I$ cannot be $L^p$-bounded for $p<\frac{r+1}{r}$.

Now let $p(r)=\frac{r+1}{r}$. To disprove $L^{p(r)}$-boundedness  we define  for $N\gg\log_2\tfrac 1\eps$
\[F_N= \sum_{\log_2\tfrac 1\eps<j\le N} 2^{jr} g_{2^{-j}}.\] Since $\|g_{2^{-j}}\|_{p(r)}\lc c_\eps 2^{-jr}$   and the sets $R_{2^{-j}}$ are disjoint we  get
\[\|F_N\|_{p(r)}  \lc   N^{1/p(r)}.\]
On the other hand,  the above lower bounds show 
$A_{t_x}F_N(x) \gc N$ for $x\in U_\eps$ and thus  
\[\frac{\| \fM_I F_N\|_{p(r)}}{ \|F_N\|_{p(r)}}  \gc  N^{1/(r+1)}.\] Letting $N\to\infty$ we see that   $L^{p(r)}$-boundedness  fails. \qed

\section{Unboundedness of the maximal operator in the case $r=1$} \label{sec:r=1}
By the reduction in \S\ref{sec:reductiontom=1} we may assume $m=1$. If $r=1$ then $J=0$.  Let $\la$ 
be a nonzero vector in $\bbR^{d}$. We prove a more general result, replacing the surface  measure of $\Sigma$ by a more general finite (positive)  Borel measure $\mu$ and consider 
\[ \Gamma_t f( \underline x,x_{d+1})= \int f(\underline x-t\underline y, x_{d+1}-t^2 \la^\intercal \underline y) d\mu(\ubar y).\]
Note that for the  special case of  $\mu$ as in Theorem \ref{thm:main}  this is the case $J=0$ in \eqref{eq:Atg}. 
$\Gamma_t f$ is well-defined for  continuous functions $f$ with compact support, and $\Gamma_t f$ is then a continuous function.  For a compact interval   $I\subset (0,\infty) $  the maximal function 
 \[ M_I f(x)= \sup_{t\in I} |\Gamma_tf(x)| \]
is then well defined as a  Borel measurable function in $\bbR^{d+1}$. We show that no nontrivial boundedness property holds for $M_I$ if 
$\mu$ is not supported in the orthogonal complement of $\la$ in $\bbR^d$.  In particular this applies to prove part (ii) of Theorem \ref{thm:main}.

\begin{prop}   Assume that $\la\neq 0$ and that $\mu$ is  a finite positive Borel measure in $\bbR^d$, with the property that $\mu ((\la^\perp)^\complement)>0$.
Suppose that there is a positive constant $C$   such that the  inequality 
\[\|M_I f\|_p\le C\|f\|_p\] holds for all characteristic functions of open sets with compact closure. Then $p=\infty$.
\end{prop}

\begin{proof} 
In what follows we will work with nonnegative functions throughout, so we may reduce the length of the interval for lower bounds. By parabolic scaling (\S\ref{sec:parabolicscaling}) and possible shrinking the $t$-interval  we may assume 
\[ \text{ $I=[a,1]$ where $a= \tan (\tfrac \pi 4(1-\tfrac 1N))$,}\]  for some $N\ge 6$. 

Let $R$  be a rotation in $\bbR^d$ such that $Re_d=\la/\|\la\|$. Let $\mu_R=\mu(R\cdot)$, formally defined by  $\int u(\ubar y) d\mu_R(\ubar y)= \int u(R^{-1} \ubar z) d\mu(\ubar z)$ for test functions $u$.
Then if 
\[\cA_t f(x):= \int f(\ubar  x-t\ubar  y, x_{d+1}-t^2 e_d^\intercal \ubar  y) d\mu_R (\ubar  y)\]
we have 
\begin{multline*}\Gamma_t  f(\ubar  x, x_{d+1})= \cA_t F_{R,\la } ( R^{-1}   \ubar  x, \|\la \|^{-1} x_{d+1} ) \\ \text{ with } \quad F_{R,\la}(\ubar  w, w_{d+1})= f(R \ubar  w, \|\la \|w_{d+1}). \end{multline*}
We also note that the assumption  of $\supp (\mu)$  not contained in $\la^\perp$  is equivalent with $\supp(\mu_R)$ not contained in $e_d^\perp$. 
Hence  $f\mapsto \sup_{ t\in[a,1]}|\Ga_t f|$ is bounded on $L^p$ if and only if   $f\mapsto \sup_{t\in[a,1]} |\cA_t f |$ is bounded on $L^p$.

{\it Proof of unboundedness for $p<\infty$.} We split variables in $\bbR^d$ as 
 $\ubar  y=(y',y_d)$.
Since the measure $\mu_R$ is not supported in $e_d^\perp$ there exists a bounded open subset $W\subset\{\ubar  y\in \bbR^d: y_d\neq 0\}$ such that $\mu_R(W)>c>0$.  Let $L>0$ be such that $|y'|\le L$ for all $y=(y',y_d)\in W$. When deriving lower bounds on nonnegative functions we will  replace $\mu_R$ by $\chi \mu_R$ where $\chi$ is continuous and compactly supported in $W$ and such  that
\[\int \chi(\ubar  w) d\mu_R(\ubar w)>c/2.\]

We use the Nikodym set in $\bbR^2$ to construct our counterexample. The original complicated construction by Nikodym appeared in \cite{Nikodym}. Subsequent   constructions are simpler and based on the Perron tree constructions used for the Kakeya set, see de Guzm\'an's book \cite{Guzman1}.
According to \cite[Theorem 3.4]{Guzman1} there exists a measurable set  $F_0\subset [-1,1]^2$  of full Lebesgue measure $|F_0|=4$ and  a Lebesgue null set $E_0\subset \bbR^2$  such that for each $w\in F_0$ there is a straight line \[l(w)= w+ 
\{r(\cos(\alpha(w)),\sin(\alpha(w))): r\in \bbR\} \] with   $l(w)\setminus \{w\}\subset E_0$; here   $\alpha(w)\in [0,\pi)$ and $w\mapsto \alpha(w)$ is a measurable function.
Let $\widetilde F=\{w\in F_0:|w|\le 1\}$.

By pigeonholing  there exists $k\in\{0,1,\dots,4N-1\}$ so that 
\[  \widetilde F_{k}=\{w\in \wt F: \alpha(w)\in [\tfrac{k\pi}{4N}, \tfrac{(k+1)\pi}{4N})\}\] satisfies 
\[|\widetilde F_{k}|\ge |\widetilde F|/4N=\pi/4N .\]
Let $\rho_\beta $ be the  planar rotation $\begin{pmatrix} \cos\beta &-\sin\beta\\ \sin\beta&\cos\beta\end{pmatrix}$ 
with $\beta=-\tfrac{k+1}{4N}\pi +\frac\pi 4$.
Let $F= \rho_\beta \widetilde F_{k} $ and $E= \rho_\beta E_0$.

Then $|F|\ge \pi/4N$ and $E$ is a Lebesgue null set. By construction  there is, for every $w\in F$, a line
$\ell(w)= \{(r\cos \ga(w), r\sin \ga(w)) :r\in \bbR\}$ such that $\ga(w)\in [(1-\frac 1N) \frac{\pi}{4}, \frac\pi 4)$ and $\ell(w) \setminus\{w\} \subset E$.

For each $w\in F$ let $s(w)$ be the slope of the line $\ell(w)$ and then 
\[ (s(w), s(w)^2)= s(w) \sqrt{1+s(w)^2} (\cos\gamma(w),\sin\gamma(w)),\]
and 
$\gamma(w)= \arctan s(w)$. 
In particular 
\[\ell(w)\setminus \{w\}=w+\{ (r s(w), rs (w)^2): r\neq 0\}.\]

We return to the task of deriving 
lower bounds for the maximal function $\sup_{t\in [a,1]}|\cA_t f|$. 
 Let 
\[\cE= \{(y',y_d,y_{d+1})\in \bbR^{d+1}:  \,|y'|\le 2+L,\, (y_d,y_{d+1})\in E \}.
\]
$\cE$ is a Lebesgue null set and thus for small $\eps>0$ there is an open subset $\cV_\eps$ 
of measure $<\eps$ with $\cV_\eps\supset\cE.$ Let $f_\eps=\bbone_{\cV_\eps}$, then  
clearly $\|f_\eps\|_p\le\eps^{1/p}$.

Let
\[\cF=\{(x',x_d,x_{d+1})\in \bbR^{d+1}: |x'|\le 1,\, (x_d, x_{d+1})\in F \}\] so that $\cF$ has positive Lebesgue measure. 
Let
\[t_x:=s(x_d,x_{d+1}).\]
For $x\in \cF$ we have
\[\cA_{t_x} f_\eps(x) \ge  \int f_\eps( x'-t_x y', x_d-t_x y_d,  x_{d+1} -t_x^2 y_d) \chi(y',y_d) d\mu_R(\ubar  y).
\]
On the support of $\chi$ we have $y_d\neq 0$. The vector 
 \[\big( x_d-t_x y_d, 
x_{d+1}-t_x^2 y_d\big) \]
thus 
belongs to $\ell(x_d,x_{d+1})\setminus \{(x_d,x_{d+1})\}\subset  E $ and since $|x'-t_x y'|\le 2+L$  we get
\[ \sup_{a\le t\le 1}\cA_t f_\eps(x) \ge \cA_{t_x}  f_\eps(x) \ge \int \chi(y',y_d) d\mu_R(\ubar  y) >c/2
\] 
for all $\eps>0$. 
If $M_I$ were   bounded on $L^p$ there would be a constant $C>0$ such that
$0<|\cF|^{1/p} c/2 \le C\eps^{1/p}$  which for sufficiently small $\eps>0$ cannot hold  if  $p<\infty$.
\end{proof}

\newpage
\end{document}